\documentclass[reqno]{amsart}

\usepackage{mathrsfs}
\usepackage{amsmath}
\usepackage{amssymb}
\usepackage{cite}
\usepackage{latexsym}
\usepackage{graphicx}

\usepackage{color}
\usepackage{comment}

\theoremstyle{plain}
\newtheorem{thm}{Theorem}[section]
\newtheorem{lem}[thm]{Lemma}

\newtheorem{dfn}[thm]{Definition}
\newtheorem{prop}[thm]{Proposition}
\newtheorem{rmk}[thm]{Remark}

\def\C{\mathscr{C}}
\def\D{\mathrm{D}}
\def\G{\mathscr{G}}

\def\T{\mathrm{T}}

\def\d{\mathrm{d}}
\def\h{\mathrm{h}}

\def\s{\mathrm{s}}
\def\u{\mathrm{u}}

\def\Cset{\mathbb{C}}
\def\Kset{\mathbb{K}}
\def\Lset{\mathbb{L}}
\def\Nset{\mathbb{N}}
\def\Rset{\mathbb{R}}
\def\Sset{\mathbb{S}}

\def\GL{\mathrm{GL}}

\def\id{\mathrm{id}}

\def\epsilon{\varepsilon}

\DeclareMathOperator{\cn}{cn}
\DeclareMathOperator{\sech}{sech}
\DeclareMathOperator{\cosech}{cosech}

\makeatletter
 \@addtoreset{equation}{section}
\makeatother

\begin{document}

% **********************************************************
% Title of This Paper
% **********************************************************

\title[Nonintegrability near homo- and heteroclinic orbits]%
{Nonintegrability of time-periodic perturbations of single-degree-of-freedom Hamiltonian systems
 near homo- and heteroclinic orbits}

\author{Kazuyuki Yagasaki}

\address{Department of Applied Mathematics and Physics, Graduate School of Informatics,
Kyoto University, Yoshida-Honmachi, Sakyo-ku, Kyoto 606-8501, JAPAN}
\email{yagasaki@amp.i.kyoto-u.ac.jp}

\date{\today}

\begin{abstract}
We consider time-periodic perturbations of single-degree-of-\linebreak
 freedom Hamiltonian systems
 and study their \emph{real-meromorphic} nonintegrability in the Bogoyavlenskij sense
 using a generalized version due to Ayoul and Zung of the Morales-Ramis theory.
The perturbation terms are assumed to have finite Fourier series in time,
 and the perturbed systems are rewritten as higher-dimensional autonomous systems
 having the small parameter as a state variable.
We show that if the Melnikov functions are not constant,
 then the autonomous systems are not \emph{real-meromorphically} integrable
 near homo- and heteroclinic orbits.
Our result is not just an extension of previous results for homocliic orbits to heteroclinic orbits
 and provides a stronger conclusion than them for the case of homoclinic orbits.
We illustrate the theory for two periodically forced Duffing oscillators
 and a periodically forced two-dimensional system.
\end{abstract}
\maketitle

% **********************************************************
% Section 1
% **********************************************************

\section{Introduction}

\begin{figure}[t]
\includegraphics[scale=1]{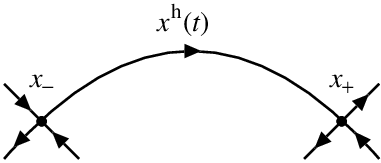}
\caption{Assumptions~(A1) and (A2).
\label{fig:1a}}
\end{figure}

In this paper we study the nonintegrability of systems of the form
\begin{equation}
\dot{x}=J\D H(x)+\epsilon g(x,\omega t),\quad
x\in\Rset^2,
\label{eqn:syse}
\end{equation}
where $\epsilon$ is a small parameter such that $0<|\epsilon|\ll 1$, $\omega>0$ is a constant,
 $H:\Rset^2\to\Rset$ and $g:\Rset^2\times\Sset^1$ are analytic,
 and $J$ is the $2\times 2$ symplectic matrix,
\[
J=
\begin{pmatrix}
0 & 1\\
-1 & 0
\end{pmatrix}.
\]
When $\epsilon= 0$, the system \eqref{eqn:syse} becomes a planar Hamiltonian system
\begin{equation}
\dot{x}=J\D H(x)
\label{eqn:sys0}
\end{equation}
with a Hamiltonian function $H(x)$.
Thus, the system \eqref{eqn:syse} represents a time-periodic perturbation
 of the single-degree-of-freedom Hamiltonian system.
We make the following assumptions on the unperturbed system \eqref{eqn:sys0}:
\begin{enumerate}
\setlength{\leftskip}{-1em}
\item[\bf(A1)]
There exist two saddles at $x=x_\pm$
 such that the Jacobian matrices $J\D^2H(x_\pm)$
 have a pair of real eigenvalues $\lambda_\pm,-\lambda_\pm$,
 where the upper or lower signs in the subscripts are taken simultaneously
 and $\lambda_\pm>0$.
\item[\bf(A2)]
The two saddles $x=x_\pm$  are connected by a heteroclinic orbit $x^\h(t)$.
See Fig.~\ref{fig:1a}.
\end{enumerate}
In assumption~(A1) we allow $x_+=x_-$.
If $x_+=x_-$, then $x^\h(t)$ is a homoclinic orbit in (A2). 

Systems of the form \eqref{eqn:syse} represent many forced nonlinear oscillators
 and have attracted much attention \cite{GH83b,W03}.
In particular, perturbation techniques called the homoclinic and subharmonic Melnikov methods 
 have been developed:
The homoclinc Melnikov method
 enables us to discuss the existence of transverse homo- and heteroclinic orbits and their bifurcations
 \cite{GH83b,M63,W03},
 and the subhamonic Melnikov method 
 to discuss the existence of periodic orbits and their stability and bifurcations
 \cite{GH83a,GH83b,W03,Y96,Y02,Y03a}.
For example, if the (homo- or heteroclinic) Melnikov function 
\begin{equation}
M(\theta)=\int_{-\infty}^\infty \D H(x^\h(t))\cdot g(x^\h(t),\omega t+\theta)\d t
\label{eqn:Mel}
\end{equation}
has a simple zero $\theta=\theta_0\in\Sset^1$, i.e.,
\[
M(\theta_0)=0,\quad
\frac{\d M}{\d\theta}(\theta_0)\neq 0,
\]
then there exist transverse homo- or heteroclinic orbits,
 depending on whether $x^\h(t)$ is a homo- or heteroclinic orbit.
In particular, the existence of transverse homoclinic orbits implies
 by the Smale-Birkhoff  homocloinic theorem \cite{GH83b,W03} that
 chaotic motions occurs in \eqref{eqn:syse}.
The techniques have been successfully applied
 to reveal the dynamics of numerous forced nonlinear oscillators.
See \cite{GH83a,GH83b,W03,Y96,Y02,Y03a} for more details.

We rewrite \eqref{eqn:syse} as an autonomous system
\begin{equation}
\dot{x}=J\D H(x)+\epsilon g(x,\theta),\quad
\dot{\theta}=\omega,\quad
(x,\theta)\in\Rset^2\times\Sset^1.
\label{eqn:asys}
\end{equation}
We extend the domains of the independent and state variables, $t$ and $x$,
 to those containing $\Rset$ and $\Rset^2$ in $\Cset$ and $\Cset^2$, respectively, if necessary.
We adopt the following definition of integrability in the Bogoyavlenskij sense \cite{B98}.
 
\begin{dfn}[Bogoyavlenskij]
\label{dfn:1a}
For $n,q\in\Nset$ such that $1\le q\le n$,
 an $n$-dimensional dynamical system
\begin{equation}
\dot{x}=f(x),\quad x\in\Rset^n\text{ or }\Cset^n,
\label{eqn:gsys}
\end{equation}
is called \emph{$(q,n-q)$-integrable} or simply \emph{integrable} 
 if there exist $q$ vector fields $f_1(x)(:= f(x)),f_2(x),\dots,f_q(x)$
 and $n-q$ scalar-valued functions $F_1(x),\dots,F_{n-q}(x)$ such that
 the following two conditions hold:
\begin{enumerate}
\setlength{\leftskip}{-1.8em}
\item[\rm(i)]
$f_1(x),\dots,f_q(x)$ are linearly independent almost everywhere and commute with each other,
 i.e., $[f_j,f_k](x):=\D f_k(x)f_j(x)-\D f_j(x)f_k(x)\equiv 0$ for $j,k=1,\ldots,q$,
 where $[\cdot,\cdot]$ denotes the Lie bracket$;$
\item[\rm(ii)]
The derivatives $\D F_1(x),\dots, \D F_{n-q}(x)$ are linearly independent almost everywhere
 and $F_1(x),\dots,F_{n-q}(x)$ are first integrals of $f_1, \dots,f_q$,
 i.e., $\D F_k(x)\cdot f_j(x)\equiv 0$ for $j=1,\ldots,q$ and $k=1,\ldots,n-q$,
 where ``$\cdot$'' represents the inner product.
\end{enumerate}
We say that the system \eqref{eqn:gsys}
 is \emph{analytically} $($resp. \emph{meromorphically}$)$ \emph{integrable}
 if the first integrals and commutative vector fields are analytic $($resp. meromorphic$)$. 
\end{dfn}
Definition~\ref{dfn:1a} is considered as a generalization of %complete
 Liouville-integrability for Hamiltonian systems \cite{A89,M99}
 since an $n$-degree-of-freedom Liouville-integrable Hamiltonian system with $n\ge 1$
 has not only $n$ functionally independent first integrals
 but also $n$ linearly independent commutative (Hamiltonian) vector fields
 generated by the first integrals.
We treat \eqref{eqn:asys} directly in the framework of Bogoyavlenskij-integrability
 even when the perturbation term $g(x,\theta)$ is Hamiltonian,
 i.e., $g(x,\theta)=J\D_x\tilde{H}(x,\theta)$
 for some function $\tilde{H}:\Rset^2\times\Sset^1\to\Rset$.
So the system \eqref{eqn:asys} is analytically $(2,1)$-integrable when $\epsilon=0$
 since $H(x)$ is an analytical first integral
 and $\partial/\partial\theta$ is an analytic commutative  vector field,
 and it may be $(q,3-q)$-integrable with $q=1,2$ or $3$ when $\epsilon>0$.

The nonintegrability of \eqref{eqn:asys} near the homoclinic orbit was studied
 by Morales-Ruiz \cite{M02} earlier and by the author and coworker \cite{MY22a} very recently.
Morales-Ruiz \cite{M02} discussed the case of Hamiltonian perturbations
 and showed a relationship of their nonintegrability
 with a version due to Ziglin \cite{Z82} of the Melnikov method, %\cite{M63}.
 which enables us to detect transversal self-intersection of complex separatrices
 of periodic orbits, unlike the standard version \cite{GH83b,M63,W03}.
More concretely, under some restrictive conditions
 (see Remarks~\ref{rmk:1a}(ii) and \ref{rmk:4a}(ii) below),
 he essentially proved that they are \emph{complex-meromorphically} nonintegrable
 in the Bogoyavlenskij sense 
 when the small parameter $\epsilon$ is taken as one of the state variables
 if the Melnikov function which is a contour integral having the same integrand as \eqref{eqn:Mel}
 along a closed path in the complex plane is not identically zero,
 based on a generalized version due to Ayoul and Zung \cite{AZ10}
 of the Morales-Ramis theory \cite{M99,MR01}.
The generalized theory says that the system \eqref{eqn:gsys} is Bogoyavlenskij-nonintegrable
 near a particular nonconstant solution
 if the identity component of the diﬀerential Galois group for the variational equation (VE),
 i.e., the linearized equation, of \eqref{eqn:gsys} along the solution is not commutative.
See Section~2 for more details.
 
On the other hand, Motonaga and Yagasaki \cite{MY22a} developed a technique
 which allows us to prove the real-analytic nonintegrability
 of nearly integrable dynamical systems containing \eqref{eqn:asys} was developed,
 based on the results of \cite{MY21}.
In particular, they showed that if the Melnikov function \eqref{eqn:Mel} is not constant,
 then the system \eqref{eqn:asys} is not \emph{real-analytically} integrable
 in a region near the homoclinic orbit
 such that the first integrals and commutative vector fields also depend real-analytically
 on $\epsilon$ near $\epsilon=0$, when in \eqref{eqn:sys0}
 there exists a one-parameter family of periodic orbits which converge to the homoclinic orbit
 as their periods tend to infinity.
Note that the results of \cite{M02,MY22a} do not apply immediately
 when $x^\h(t)$ is a heteroclinic orbit.

Moreover, the author developed a technique which permits us
 to prove \emph{complex-meromorphic} nonintegrability
 of nearly integrable dynamical systems
 near resonant periodic orbits in \cite{Y22c,Y22d},
 based on the generalized version of the Morales-Ramis theory
 and its extension, the Morales-Ramis-Sim\'o theory \cite{MRS07}.
In particular, he showed in \cite{Y22c} that if a contour integral
 which is similar to the Melnikov function in \cite{M02}
 but depend on the unperturbed resonant periodic orbit is not zero,
 then the system \eqref{eqn:asys} is not complex-meromorphically integrable
 near the periodic orbit
 such that the first integrals and commutative vector fields
 also depend complex-meromorphically on $\epsilon$ near $\epsilon=0$,
 when in \eqref{eqn:sys0} there exists a one-parameter family of periodic orbits.
These techniques were successfully applied to the Duffing oscillators,
 one of which is the same as Eq.~\eqref{eqn:ex1} below, in \cite{M02,MY22a,Y22c}
 and to a forced pendulum in \cite{MY22b}.

We now state our main result.
We additionally assume the following:
\begin{enumerate}
\setlength{\leftskip}{-1em}
\item[\bf(A3)]
The perturbation term $g(x,\theta)$ has a finite Fourier series, i.e.,
\[
g(x,\theta)
 =\sum_{j=-N}^N\hat{g}_j(x)e^{ij\theta},
\]
where $N\in\Nset$ and $\hat{g}_j(x)$, $j=-N,\ldots,N$, are meromorphic.
\end{enumerate}

\begin{rmk}\
\label{rmk:1a}
\begin{enumerate}
\setlength{\leftskip}{-1.8em}
\item[(i)]
We require that
 $\hat{g}_j(x)$, $j=-N,\ldots,N$, are analytic on $\Rset^2$,
 since $g(x,\theta)$ is analytic on $\Rset^2\times\Sset^1$.
\item[(ii)]
Morales-Ruiz {\rm\cite{M02}} assumed that
 $g(x,\theta)$ is represented by rational functions of $e^{i\theta}$
 with meromorphic coefficients of $x$
 but that $x^\h(t)$ is a homoclinic orbit expressed
 by rational functions of an exponential function, say $e^{\lambda t}$ with $\lambda>0$.
See Assumption $(1)$ of {\rm\cite{M02}}.
\end{enumerate}
\end{rmk}

Since $g(x,\theta)$ must be real on $\Rset^2\times\Sset^1$ and
\[
\hat{g}_j(x)=\frac{1}{2\pi}\int_0^{2\pi}g(x,\theta)e^{-ij\theta}\d\theta,
\]
we have $\hat{g}_j^\ast(x)=\hat{g}_{-j}(x)$,
 where the superscript `$\ast$' represents complex conjugate.
Under assumption (A3) we rewrite \eqref{eqn:syse} as
\begin{equation}
\begin{split}
&
\dot{x}=J\D H(x)+\epsilon a_0(x)+\sum_{j=1}^N(a_j(x)u_j+b_j(x)v_j),\\
&
\dot{\epsilon}=0,\quad
\dot{u}_j=-j\omega v_j,\quad
\dot{v}_j=j\omega u_j,\quad
j=1,\ldots,N,
\end{split}
\label{eqn:rsys}
\end{equation}
where
\begin{align*}
&
a_0(x)=\hat{g}_0(x),\quad
a_j(x)=\hat{g}_j(x)+\hat{g}_{-j}(x),\\
&
b_j(x)=i(\hat{g}_j(x)-\hat{g}_{-j}(x)),\quad
j=1,\ldots,N.
\end{align*}
Note that $a_0(x)$, $a_j(x)$ and $b_j(x)$, $j=1,\ldots,N$, are real for $x\in\Rset^2$,
 and that $(u_j,v_j)=(\epsilon\cos j\omega t,\epsilon\sin j\omega t)$ is a solution
 to the $(u_j,v_j)$-components of \eqref{eqn:rsys} for $j=1,\ldots,N$.
Let $u=(u_1,\ldots,u_N)$ and $v=(v_1,\ldots,v_N)$.
Our main theorem is stated as follows.

\begin{thm}
\label{thm:main}
Suppose that the Melnikov function $M(\theta)$ is not constant under assumptions~{\rm(A1)-(A3)}.
Then the system \eqref{eqn:rsys} is not real-meromorphically integrable near
\[
\hat{\Gamma}
=\{(x,\epsilon,u,v)=(x^\h(t),0,0,0)\in\Rset^2\times\Rset\times\Rset^N\times\Rset^N\mid t\in\Rset\}
\cup\{(x_\pm,0,0,0)\}
\]
in $\Rset^{2N+3}$.
\end{thm}
%\newpage

\begin{rmk}\
\label{rmk:1b}
\begin{enumerate}
\setlength{\leftskip}{-1.8em}
\item[(i)]
The homoclinic Melnikov method {\rm\cite{GH83b,M63,W03}}
 does not yield such a conclusion as in Theorem~$\ref{thm:main}$
 even if the Melnikov function $M(\theta)$ has a simple zero,
 since it says nothing about the nonexistence of commutative vector fields
 although no additional \emph{real-analytic} first integral exists
 $($see, e.g., Theorem~$3.10$ of {\rm\cite{M73})}.
\item[(ii)]
In {\rm\cite{Y03b,YY19}} the nonintegrability of two-degree-freedom Hamiltonian systems of the form
\begin{align*}
&
\dot{x}=J\D H(x)+J\D_x H_1(x,u_1,v_1),\\
&
\dot{u}_1=-\omega v_1-\D_{v_1}H_1(x,u_1,v_1),\\
&
\dot{v}_1=\omega u_1+ \D_{u_1}H_1(x,u_1,v_1),
\end{align*}
with the Hamiltonian $H(x)+\tfrac{1}{2}\omega^2(u_1^2+v_1^2)+H_1(x,u_1,u_2)$,
 where $H_1(x,0,0)\equiv 0$,
 near homoclinic or heteroclinic orbits on the $x$-plane were studied.
Theorem~$\ref{thm:main}$ states that even simpler systems
  having no coupling terms in their $(u_1,v_1)$-components may be nonintegrable.
\end{enumerate}
\end{rmk}

We emphasize that
 Theorem~$\ref{thm:main}$ is immediately applicable
 when $x^\h(t)$ is a heteroclinic orbit,
 while the previous results of \cite{M02,MY22a} are not.
Moreover, it guarantees the \emph{real-meromorphic} nonintegrability of \eqref{eqn:rsys}
 when $M(\theta)$ is not constant.
Our proof of Theorem~\ref{thm:main}
 is also based on the generalized version of the Morales-Ramis theory,
 but the approach is very different from \cite{M02}.
The arguments used here are rather similar to those of \cite{MP99,Y03b,YY19},
 in which the Liouville-integrability of two-degree-of-freedom Hamiltonian systems
 near homo- or heteroclinic orbits to saddle-center equilibria was discussed,
 although very different situations were considered there.
Similar arguments were also used in different contexts in \cite{BY12a,Y22a,Y22b,YY17}
 (see Sections~3 and 4).
We illustrate our theory for two periodically forced Duffing oscillators:
\begin{equation}
\dot{x}_1=x_2,\quad
\dot{x}_2=x_1-x_1^3+\epsilon(\beta\cos\omega t-\delta x_2)
\label{eqn:ex1}
\end{equation}
and
\begin{equation}
\dot{x}_1=x_2,\quad
\dot{x}_2=-x_1+x_1^3+\epsilon(\beta\cos\omega t-\delta x_2);
\label{eqn:ex2}
\end{equation}
and for a two-dimensional system with periodic forcing:
\begin{equation}
\dot{x}_1=x_1-x_1^2,\quad
\dot{x}_2=-x_2+2x_1x_2+\epsilon(\beta\cos\omega t-\delta x_2),
\label{eqn:ex3}
\end{equation}
where $\beta,\omega>0$ and $\delta\ge 0$ are constants.
The unperturbed system has a pair of homoclinic orbits
\begin{equation}
x^\h(t)=(\sqrt{2}\sech t,-\sqrt{2}\sech t \tanh t)
\label{eqn:ex1h}
\end{equation}
to the equilibrium at $x=0$ for \eqref{eqn:ex1}, a pair of heteroclinic orbits
\begin{equation}
x^\h(t)=\left(\pm\tanh(t/\sqrt{2}),\pm\sech^2(t/\sqrt{2})/\sqrt{2}\right)
\label{eqn:ex2h}
\end{equation}
connecting two equilibria $x=(-1,0)$ and $(1,0)$ for \eqref{eqn:ex2},
 and a heteroclinic orbit
\begin{equation}
x^\h(t)=\left(\frac{e^t}{e^t+1},0\right)
\label{eqn:ex2h}
\end{equation}
connecting two equilibria $x=(0,0)$ and $(1,0)$ for \eqref{eqn:ex3}.
In particular, there exists a heteroclinic orbit but no periodic orbit
 in the unperturbed system for \eqref{eqn:ex3}.
See Fig.~\ref{fig:5c} in Section~5 for the phase portrait of \eqref{eqn:ex3} with $\epsilon=0$.
The techniques of \cite{MY22a,Y22c,Y22d} can be applied to the first two examples
 since there are one-parameter families of periodic orbits in the perturbed systems
 even near the homo- and heteroclinic orbits (see also \cite{MY22b}), 
 but cannot to the third example.

The outline of this paper is as follows:
In Section~2 we briefly review the fundamental theories,
 the differential Galois theory and generalized Morales-Ramis theory.
We provide preliminary results in Section~3,
 and prove Theorem~\ref{thm:main} in Section~4.
Finally we  present the two examples in Section~5.

% **********************************************************
% Section 2
% **********************************************************

\section{Fundamental Theories}
In this section we provide outlines of 
 the differential Galois theory and generalized  Morales theory.

\subsection{Differential Galois theory}

We begin with the differential Galois theory for linear differential equations.
See the textbooks \cite{CH11,PS03} for more details on the theory.

Consider a linear system of differential equations
\begin{equation}
\dot{x}=Ax,\quad A\in\mathrm{gl}(n,\Kset),
\label{eqn:lsys}
\end{equation}
where $\Kset$ is a differential field and
 $\mathrm{gl}(n,\Kset)$ denotes the ring of $n\times n$ matrices
 with entries in $\Kset$.
Here a \emph{differential field} is a field
 endowed with a derivation $\partial$,
 which is an additive endomorphism satisfying the Leibniz rule
 and represented by the overdot in \eqref{eqn:lsys}.
The set $\mathrm{C}_{\Kset}$ of elements of $\Kset$ for which $\partial$ vanishes
 is a subfield of $\Kset$
 and called the \emph{field of constants of $\Kset$}.
In our application of the theory in this paper,
 the field of constants is $\Cset$, which is algebraically closed.

A \emph{differential field extension} $\Lset\supset \Kset$
 is a field extension such that $\Lset$ is also a differential field
 and the derivations on $\Lset$ and $\Kset$ coincide on $\Kset$.
A differential field extension $\Lset\supset \Kset$
 satisfying the following two conditions is called a \emph{Picard-Vessiot extension}
 for \eqref{eqn:lsys}:
\begin{enumerate}
%\item[\bf (PV1)]
%There exists a fundamental matrix $\Xi(x)$ of \eqref{LinearSystem} with entries in $\Lset$;
\item[\bf (PV1)]
The field $\Lset$ is generated by $\Kset$
 and elements of a fundamental matrix of \eqref{eqn:lsys};
\item[\bf (PV2)]
The fields of constants for $\Lset$ and $\Kset$ coincide.
\end{enumerate}
The system \eqref{eqn:lsys}
 admits a Picard-Vessiot extension which is unique up to isomorphism.

We now fix a Picard-Vessiot extension $\Lset\supset \Kset$
 and fundamental matrix $\Phi$ with entries in $\Lset$
 for \eqref{eqn:lsys}.
Let $\sigma$ be a \emph{$\Kset$-automorphism} of $\Lset$,
 which is a field automorphism of $\Lset$
 that commutes with the derivation of $\Lset$
 and leaves $\Kset$ pointwise fixed.
Obviously, $\sigma(\Phi)$ is also a fundamental matrix of \eqref{eqn:lsys}
 and consequently there is a matrix $M_\sigma$ with constant entries
 such that $\sigma(\Phi)=\Phi M_\sigma$.
This relation gives a faithful representation
 of the group of $\Kset$-automorphisms of $\Lset$
 on the general linear group as
\[
R\colon \mathrm{Aut}_{\Kset}(\Lset)\to\GL(n,\mathrm{C}_{\Lset}),
\quad \sigma\mapsto M_{\sigma},
\]
where $\GL(n,\mathrm{C}_{\Lset})$
is the group of $n\times n$ invertible matrices with entries in $\mathrm{C}_{\Lset}$.
The image of $R$
 is a linear algebraic subgroup of $\GL(n,\mathrm{C}_{\Lset})$,
 which is called the \emph{differential Galois group} of \eqref{eqn:lsys}
 and often denoted by $\mathrm{Gal}(\Lset/\Kset)$.
This representation is not unique
 and depends on the choice of the fundamental matrix $\Phi$,
 but a different fundamental matrix only gives rise to a conjugated representation.
Thus, the differential Galois group is unique up to conjugation
 as an algebraic subgroup of the general linear group.

Let $\G\subset\GL(n,\mathrm{C}_{\Lset})$ be an algebraic group.
Then it contains a unique maximal connected algebraic subgroup $\G^0$,
 which is called the \emph{connected component of the identity}
 or \emph{identity component}.
The identity component $\G^0\subset\G$ is
% a normal algebraic subgroup and
 the smallest subgroup of finite index, i.e., the quotient group $\G/\G^0$ is finite.
 
%Let the linear system \eqref{eqn:lsys} be defined on a Riemann surface $\C$.

Let $\Kset$ be the field of meromorphic functions on a Riemann surface $\C$,
 and consider the linear system \eqref{eqn:lsys}.
A point $\bar{t}\in\C$ is called a \emph{singular point} if $A$ is not bounded at $x$.
A singular point $\bar{t}$ is called \emph{regular}
 if for any sector $\kappa_1<\arg(t-\bar{t})<\kappa_2$ with $\kappa_1<\kappa_2$
 there exists a fundamental matrix $\Phi(t)=(\phi_{ij}(t))$
 such that for some $C>0$ and integer $N$
 $|\phi_{ij}|<C|t-\bar{t}|^N$ as $t\to\bar{t}$ in the sector;
 otherwise it is called \emph{irregular}.
Let $t_0\in\C$ be a nonsingular point for \eqref{eqn:lsys}.
We prolong the fundamental matrix $\Phi(t)$ analytically
 along any loop $\gamma$ based at $t_0$ and containing no singular point,
 and obtain another fundamental matrix $\gamma*\Phi(t)$.
So there exists a constant nonsingular matrix $M_{[\gamma]}$ such that
\[
\gamma*\Phi(t) = \Phi(t)M_{[\gamma]}.
\]
The matrix $M_{[\gamma]}$ depends on the homotopy class $[\gamma]$
 of the loop $\gamma$
 and is called the \emph{monodromy matrix} of $[\gamma]$.
Let $\Lset$ be a Picard-Vessiot extension of \eqref{eqn:lsys}
 and let $\mathrm{Gal}(\Lset/\Kset)$ be the differential Galois group.
Since analytic continuation commutes with differentiation,
 we have $M_{[\gamma]}\in\mathrm{Gal}(\Lset/\Kset)$.

\subsection{Generalized Morales-Ramis theory}
We next briefly review the Morales-Ramis theory for the general system \eqref{eqn:gsys}
 in a necessary setting.
See \cite{AZ10,M99,MR01} for more details on the theory.

Consider the general system \eqref{eqn:gsys}.
Let $x=\phi(t)$ be its nonconstant particular solution.
The VE of \eqref{eqn:gsys} along $x=\phi(t)$ is given by
\begin{equation}
\dot{\xi} = \D f(\phi(t))\xi,\quad \xi \in \Cset^n.
\label{eqn:gve}
\end{equation}
Let $\C$ be a curve given by $x=\phi(t)$.
We take the meromorphic function field on $\C$
 as the coefficient field $\Kset$ of \eqref{eqn:gve}.
Using arguments given by Morales-Ruiz and Ramis \cite{M99,MR01}
 and Ayoul and Zung \cite{AZ10}, we have the following result.

\begin{thm}
\label{thm:MR}
Let $\G$ be the differential Galois group of \eqref{eqn:gve}.
%Suppose that the VE \eqref{eqn:gve} has irregular singularity at infinities in the phase space.
If the system~\eqref{eqn:gsys} is meromorphically integrable near $\C$,
 then the identity component $\G^0$ of $\G$ is commutative.
\end{thm}

By contraposition of Theorem~\ref{thm:MR}, if $\G^0$ is not commutative,
 then the system~\eqref{eqn:gsys} is meromorphically nonintegrable near $\C$.

% **********************************************************
% Section 3
% **********************************************************

\section{Preliminaries}

In this section we give preliminary results for the proof of Theorem~\ref{thm:main}.
Letting $y_0=\epsilon$ and
\[
y_j=\tfrac{1}{2}(u_j+i v_j),\quad
y_{-j}=\tfrac{1}{2}(u_j-i v_j),\quad
j=1,\ldots,N,
\]
we rewrite \eqref{eqn:rsys} as
\begin{equation}
\dot{x}=J\D H(x)+\sum_{j=-N}^N\hat{g}_j(x)y_j,\quad
\dot{y}_j=ij\omega y_j,\quad
j=-N,\ldots,N.
\label{eqn:csys}
\end{equation}
We easily see that
 $y_j=\tfrac{1}{2}\epsilon e^{ij\omega t}$ is a solution to the $y_j$-component of \eqref{eqn:csys}
 for $j\in\{-N,\ldots,N\}\setminus\{0\}$.
Let $y=(y_{-N},\ldots,y_N)$.

\begin{lem}
\label{lem:3a}
If the system \eqref{eqn:rsys} is real-meromorphically integrable
 near $\hat{\Gamma}$ in $\Rset^{2N+3}$,
 then the system \eqref{eqn:csys} is complex-meromorphically integrable near
\[
\tilde{\Gamma}
 =\{(x,y)=(x^\h(t),0)\in\Rset^2\times\Rset^{2N+1}\mid t\in\Rset\}\cup\{(x_\pm,0)\}
\]
in $\Cset^{2N+3}$.
\end{lem}

\begin{proof}
If $F(x,u,v,\epsilon)$ and $f(x,u,v,\epsilon)$ are, respectively,
 a first integral and commutative vector field for \eqref{eqn:rsys}
 near $\hat{\Gamma}$ in $\Rset^{2N+3}$,
 then so are they for \eqref{eqn:rsys} near $\hat{\Gamma}$ in $\Cset^{2N+3}$,
 and consequently so are $F(x,\tilde{u},\tilde{v},y_0)$ and $f(x,\tilde{u},\tilde{v},\epsilon)$
 for \eqref{eqn:csys} near $\tilde{\Gamma}$,
 where $\tilde{u}=(y_1+y_{-1},\ldots,y_N+y_{-N})$ and $i\tilde{v}=(y_1-y_{-1},\ldots,y_N-y_{-N})$.
Thus, we obtain the desired result.
\end{proof}

We easily see that the Melnikov function $M(\theta)$ is not constant if and only if
\begin{equation}
\hat{M}_j:=\int_{-\infty}^\infty\D H(x^\h(t))\cdot\hat{g}_j(x^\h(t))e^{ij\omega t}\d t\neq 0
\label{eqn:Mj}
\end{equation}
for one of $j=\{-N,\ldots,N\}\setminus\{0\}$.
By Lemma~\ref{lem:3a}, to prove Theorem~\ref{thm:main},
 we only have to show that
 the system \eqref{eqn:csys} is complex-meromorphically integrable near $\tilde{\Gamma}$
 if $\hat{M}_j\neq 0$ for some $j\neq 0$.
%Since $\hat{g}_j(x)^\ast=\hat{g}_{-j}(x)$, we have $\hat{M}_j^\ast=\hat{M}_{-j}$.
%In particular, if $\hat{M}_j\neq 0$, then $\hat{M}_{-j}\neq 0$.

We apply Theorem~\ref{thm:MR}
 to the nonconstant particular solution $(x,y)=(x^\h(t),0)$ in \eqref{eqn:csys}.
The VE of \eqref{eqn:csys} along the solution is given by
\begin{equation}
\dot{\xi}=J\D^2H(x^\h(t))\xi+\sum_{j=-N}^N\hat{g}_j(x^\h(t))\eta_j,\quad
\dot{\eta}_j=ij\omega\eta_j,\quad
j=-N,\ldots,N.
\label{eqn:veh}
\end{equation}
Obviously, we have the following.

\begin{lem}
\label{lem:3b}
If the differential Galois group of \eqref{eqn:veh} is commutative,
 then so are those of its $(\xi,\eta_\ell)$-components with $\eta_j=0$, $j\neq\ell$,
\begin{equation}
\dot{\xi}=J\D^2H(x^\h(t))\xi+\hat{g}_\ell(x^\h(t))\eta_\ell,\quad
\dot{\eta}_\ell=i\ell\omega\eta_\ell,
\label{eqn:vehl}
\end{equation}
for $\ell=-N,\ldots,N$.
\end{lem}

Based on Theorem~\ref{thm:MR} and Lemmas~\ref{lem:3a} and \ref{lem:3b},
 we only have to show that the differential Galois group of \eqref{eqn:vehl} is not commutative
 if $\hat{M}_\ell\neq 0$ for some $\ell\neq 0$,
 to prove Theorem~\ref{thm:main}.

Assume that $\D_{x_2}H(x^\h(t))\not\equiv 0$.
Let
\begin{equation}
X(t)=
\begin{pmatrix}
J\D H(x^\h(t)) & \chi(t)J\D H(x^\h(t))
\end{pmatrix}
+\begin{pmatrix}
0 & 0\\
0 & \D_{x_2}H(x^\h(t))^{-1}
\end{pmatrix},
\label{eqn:X}
\end{equation}
where $\chi(t)$ is a primitive function of $\D_{x_2}^2H(x^\h(t))/\D_{x_2}H(x^\h(t))^2$:
\[
\chi(t)=\int\frac{\D_{x_2}^2H(x^\h(t))}{\D_{x_2}H(x^\h(t))^2}\d t.
\]
 
\begin{lem}
\label{lem:3c}
$X(t)$ is a fundamental matrix of the $\xi$--component of \eqref{eqn:vehl} with $\eta_\ell=0$.
\end{lem}

\begin{proof}
This statement was proven in \cite{M02}.
For the reader's convenience we briefly give the proof.
Let
\[
P(t)=\left(
J\D H(x^\h(t))\
\begin{array}{c}
0\\[1ex]
\D_{x_2}H(x^\h(t))^{-1}
\end{array}
\right)
\]
and let $P(t)\Xi(t)$ be a fundamental matrix of the linear system.
Then we have
\[
\dot{P}(t)\Xi(t)+P(t)\dot{\Xi}(t)=J\D^2H(x^\h(t))P(t)\Xi(t),
\]
so that
\begin{align*}
\dot{\Xi}(t)=&P(t)^{-1}(J\D^2H(x^\h(t))P(t)-\dot{P}(t))\Xi(t)\\
=&
\begin{pmatrix}
0 & \D_{x_2}^2H(x^\h(t))/\D_{x_2}H(x^\h(t))^2\\
0 & 0
\end{pmatrix}\Xi(t).
\end{align*}
Since
\[
\Xi(t)=
\begin{pmatrix}
1 & \chi(t)\\
0 & 1
\end{pmatrix}
\]
is a fundamental matrix of the above equation, we obtain the desired result.
\end{proof}

\begin{rmk}
The second term in the right hand side of \eqref{eqn:X} can be replaced by
\[
\begin{pmatrix}
0 & \D_{x_1}H(x^\h(t))^{-1}\\
0 & 1
\end{pmatrix}
\]
with
\[
\chi(t)=\int\frac{\D_{x_1}^2H(x^\h(t))}{\D_{x_1}H(x^\h(t))^2}\d t
\]
if $\D_{x_1}H(x^\h(t))\not\equiv 0$.
If $\D_{x_2}H(x^\h(t))\equiv 0$,
 then we can apply the arguments below by this replacement.
\end{rmk}

We see that $\chi(t)=O(e^{\pm2\lambda_\pm t})$ as $t\to\pm\infty$
 since $\D_{x_2}H(x^\h(t))=O(e^{\mp\lambda_\pm t})$ and $\D_{x_2}^2H(x^\h(t))=O(1)$.
Hence, as $t\to\pm\infty$,
 the second column of $X(t)$ goes to infinity exponentially at the rate of $\pm\lambda_\pm$
 while the first column goes to zero exponentially at the rate of $\mp\lambda_\pm$.
We write
\begin{equation}
\xi_\pm=\lim_{t\to\pm\infty}J\D H(x^\h(t))e^{\pm\lambda_\pm t},\quad
\chi_\pm=\lim_{t\to\pm\infty}\chi(t)e^{\mp 2\lambda_\pm t}.
\label{eqn:lim1}
\end{equation}
We also have
\begin{equation}
\chi_\pm=\frac{\D_{x_2}^2H(x_\pm)}{2\lambda_\pm\xi_{2\pm}^2}.
\label{eqn:chi}
\end{equation}
In particular, $\xi_\pm\neq 0$ and $\chi_\pm\neq 0$
 since $\det J\D^2H(x_\pm)=\det\D^2H(x_\pm)>0$ so that $\D_{x_2}^2H(x_\pm)\neq 0$.

Let
\begin{equation}
Y(t)=\int X(t)^{-1}\hat{g}_\ell(x^\h(t))e^{i\ell\omega t}\d t,
\label{eqn:Y}
\end{equation}
of which the first element is
\[
\int\biggl(\frac{\hat{g}_{\ell 1}(x^\h(t))}{\D_{x_2}H(x^\h(t))}
-\chi(t)\D H(x^\h(t))\cdot\hat{g}_\ell(x^\h(t))\biggr)e^{i\ell\omega t}\d t
\]
where $\hat{g}_{\ell j}(x)$ is the $j$th element of $\hat{g}_\ell(x)$ for $j=1,2$,
 and the second element is
\[
\int\D H(x^\h(t))\cdot\hat{g}_\ell(x^\h(t))e^{i\ell\omega t}\d t.
\]
We easily see that
\begin{equation}
\Psi(t)=\begin{pmatrix}
X(t) & X(t)Y(t)\\
0 & e^{i\ell\omega t}
\end{pmatrix}
\label{eqn:Psi}
\end{equation}
is a fundamental matrix of \eqref{eqn:vehl}.
Let
\[
\int\D H(x^\h(t))\cdot\hat{g}_\ell(x^\h(t))e^{i\ell\omega t}\d t\to m_\pm
\]
as $t\to\pm\infty$.
From \eqref{eqn:Mj} we see that
\begin{equation}
\hat{M}_\ell=m_+-m_-.
\label{eqn:m}
\end{equation}
Using \eqref{eqn:lim1}, we also have
\begin{equation}
X(t)Y(t)=
\begin{pmatrix}
\xi_{\pm 1}\chi_\pm m_\pm\\
(\xi_{\pm 2}\chi_\pm+\xi_{\pm 1}^{-1})m_\pm
\end{pmatrix}e^{\pm\lambda_\pm t}+O(1)
\label{eqn:lim2}
\end{equation}
as $t\to\pm\infty$.

We next consider the limits of \eqref{eqn:vehl} as $t\to\pm\infty$:
\begin{equation}
\dot{\xi}=J\D^2H(x_\pm)\xi+\hat{g}_\ell(x_\pm)\eta_\ell,\quad
\dot{\eta}_\ell=i\ell\omega\eta_\ell.
\label{eqn:ve0}
\end{equation}
Let $Q_\pm$ be nonsingular matrices such that
\[
Q_\pm^{-1}J\D^2H(x_\pm)Q_\pm
=\begin{pmatrix}
\pm\lambda_\pm & 0\\
0 &\mp\lambda_\pm
\end{pmatrix}.
\]
We easily see that
\begin{equation}
\Phi_\pm(t)=\begin{pmatrix}
X_\pm(t) & X_\pm(t)Y_\pm(t)\\
0 & e^{i\omega t}
\end{pmatrix}
\label{eqn:Phi}
\end{equation}
are fundamental matrices of \eqref{eqn:ve0}, where
\begin{equation}
X_\pm(t)=Q_\pm
\begin{pmatrix}
e^{\pm\lambda_\pm t} & 0\\
0 & e^{\mp\lambda_\pm t}
\end{pmatrix}Q_\pm^{-1}
\label{eqn:X0}
\end{equation}
and
\begin{align*}
Y_\pm(t)=& \int_0^t X_\pm(-\tau)\hat{g}_\ell(x_\pm)e^{i\ell\omega t}\d\tau\\
=& Q_\pm
\begin{pmatrix}
\displaystyle
\frac{e^{(\mp\lambda_\pm+i\ell\omega)t}-1}{\mp\lambda_\pm+i\ell\omega} & 0\\
0 & \displaystyle
\frac{e^{(\pm\lambda_\pm+i\ell\omega)t}-1}{\pm\lambda_\pm+i\ell\omega}
\end{pmatrix}
Q_\pm^{-1}\hat{g}_\ell(x_\pm).
\end{align*}
Moreover, $\Phi_\pm(0)=\id_2$.

\begin{lem}
\label{lem:3d}
There exist nonsingular $2\times 2$ matrices $B_\pm$ and two-dimensional vectors $b_\pm$
 such that
\[
\lim_{t\to\pm\infty}\Psi(t)
\begin{pmatrix}
B_\pm & B_\pm b_\pm\\[0.5ex]
0 & 1
\end{pmatrix}
\Phi_\pm(t)^{-1}=\id_3,
\]
where
\begin{equation}
\begin{split}
&
\lim_{t\to\pm\infty} X(t)B_\pm X_\pm(-t)=\id_2,\\
&
\lim_{t\to\pm\infty}\left(X_\pm(t)(Y_\pm(t)-b_\pm)-X(t)Y(t)\right)=0.
\end{split}
\label{eqn:lem3d}
\end{equation}
\end{lem}

\begin{proof}
We first note that
\[
\lim_{t\to\pm\infty}J\D^2H(x^\h(t))=J\D^2H(x_\pm)
\]
since $\lim_{t\to\pm\infty}x^\h(t)=x_\pm$.
Hence, by a standard result on the asymptotic behavior of linear systems
 (e.g., Theorem~8.1 in Section~3.8 of \cite{CL55}),
 there exist fundamental matrices $\tilde{\Psi}_\pm(t)$ and $\tilde{\Phi}_\pm(t)$
 of \eqref{eqn:veh} and \eqref{eqn:ve0}, respectively, such that
\[
\lim_{t\to\pm\infty}\tilde{\Psi}_\pm(t)\tilde{\Phi}_\pm(t)^{-1}=\id_3.
\]
We can write
\[
\tilde{\Psi}_\pm(t)=\Psi(t)\tilde{C}_\pm,\quad
\Phi_\pm(t)=\tilde{\Phi}_\pm(t)\tilde{D}_\pm,
\]
where $\tilde{C}_\pm,\tilde{D}_\pm$ are certain nonsingular matrices, so that
\[
\lim_{t\to\pm\infty}\Psi(t)\tilde{C}_\pm\tilde{D}_\pm\Phi_\pm(t)^{-1}=\id_3.
\]
Using \eqref{eqn:Psi} and \eqref{eqn:Phi} and noting that
\[
\Phi_\pm(t)^{-1}=\begin{pmatrix}
X_\pm(-t) & -Y_\pm(t)e^{-i\ell\omega t}\\
0 & e^{-i\ell\omega t}
\end{pmatrix},
\]
we obtain the desired result.
\end{proof}

Similar results were also used for homo- or heteroclinic orbits to saddle-centers
 in two-degree-of-freedom Hamiltonian systems in \cite{Y00,Y03b,YY19},
 and for two-dimensional linear systems in \cite{Y22b}.
We remark that $B_\pm$ is not unique.
Actually, we easily see that
\[
\lim_{t\to\pm\infty}X(t)\left(B_\pm+\begin{pmatrix}
0 & c\\
0 & 0
\end{pmatrix}Q_\pm^{-1}\right)X_\pm(t)=\id_2
\]
for any $c\in\Cset$.
%Let $B_0=B_+^{-1}B_-$.
%Using \eqref{eqn:X}, \eqref{eqn:X0} and the first equation of \eqref{eqn:lem3d}, we obtain
%\[
%B_\pm=
%\begin{pmatrix}
%\xi_{\pm 1}^{-1} & 0\\[1ex]
%-\xi_{\pm 2} & \xi_{\pm1}
%\end{pmatrix},
%\]
%so that
%\begin{equation}
%B_0=\begin{pmatrix}
%\xi_{+1} & 0\\[1ex]
%\xi_{+2} & \xi_{+1}^{-1}
%\end{pmatrix}\begin{pmatrix}
%\xi_{-1}^{-1} & 0\\[1ex]
%-\xi_{-2} & \xi_{-1}
%\end{pmatrix}
%=\begin{pmatrix}
%\displaystyle
% \frac{\xi_{+1}}{\xi_{-1}} & 0\\[2ex]
%\displaystyle
%\frac{\xi_{+2}}{\xi_{-1}}-\frac{\xi_{-2}}{\xi_{+1}}
%& \displaystyle
% \frac{\xi_{-1}}{\xi_{+1}}
%\end{pmatrix}.
%\label{eqn:B0}
%\end{equation}
Using \eqref{eqn:lim2}, \eqref{eqn:X0} and the second equation of \eqref{eqn:lem3d},
 we have
\begin{equation}
b_\pm=-\begin{pmatrix}
\xi_{\pm 1}\chi_\pm m_\pm\\[1ex]
(\xi_{\pm 2}\chi_\pm+\xi_{\pm 1}^{-1})m_\pm
\end{pmatrix}
-c_\pm,
\label{eqn:b}
\end{equation}
where
\begin{align}
c_\pm=&Q_\pm\begin{pmatrix}
\displaystyle
\frac{1}{\mp\lambda_\pm+i\ell\omega} & 0\\
0 & \displaystyle
\frac{1}{\pm\lambda_\pm+i\ell\omega}
\end{pmatrix}Q_\pm^{-1}\hat{g}_\ell(x_\pm)\notag\\
=&-\frac{1}{\lambda_\pm^2+\ell^2\omega^2}(J\D^2 H(x_\pm)+i\ell\omega\,\id_2)\hat{g}_\ell(x_\pm).
\label{eqn:c}
\end{align}
since
\begin{align*}
&
Q_\pm^{-1}X_\pm(t)Q_\pm(Q_\pm^{-1}Y_\pm(t)-Q_\pm^{-1}b_\pm)\\
&
=\begin{pmatrix}
e^{\pm\lambda_\pm t} & 0\\
0 & e^{\mp\lambda_\pm t}
\end{pmatrix}\left(
\begin{pmatrix}
\displaystyle
\frac{e^{(\mp\lambda_\pm+i\ell\omega)t}-1}{\mp\lambda_\pm+i\ell\omega} & 0\\
0 & \displaystyle
\frac{e^{(\pm\lambda_\pm+i\ell\omega)t}-1}{\pm\lambda_\pm+i\ell\omega}
\end{pmatrix}
Q_\pm^{-1}\hat{g}_\ell(x_\pm)-Q_\pm^{-1}b_\pm\right).
\end{align*}
Note that there does not uniquely exist $b_\pm$ like $B_\pm$.

% **********************************************************
% Section 4
% **********************************************************

\section{Proof of Theorem~\ref{thm:main}}

\begin{figure}[t]
\includegraphics[scale=1]{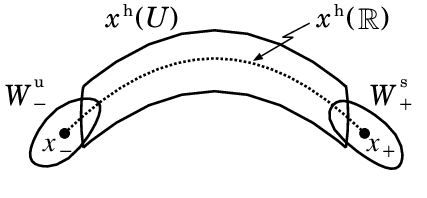}
\caption{Riemann surface $\Gamma=x^\h(U)\cup W_+^\s\cup W_-^\u$.
\label{fig:2a}}
\end{figure}

We are now in a position to prove Theorem~\ref{thm:main}.
The basic idea was previously used
 for nonintegrability of two-degree-of-freedom Hamiltonian systems or general dynamical systems
 near homo- or heteroclinic orbits in \cite{MP99,Y03b,YY17,YY19},
 for bifurcations of homoclinic orbits in \cite{BY12,Y22a},
 and for integrability of two-dimensional linear systems appearing
 in application of the inverse scattering transform
 (see, e.g., Chapter~9 of\cite{A11}) in \cite{Y22b}.

Let $W_\pm^{\s,\u}$
 be the one-dimensional local holomorphic stable and unstable manifolds of $x_\pm$.
See Section~1.7 of \cite{IY08} for the existence of such holomorphic stable and unstable manifolds.
Let $R > 0$ be sufficiently large and let $U$
 be a neighborhood of the open interval  $(-R,R) \subset \Rset$ in $\Cset$
 such that $x^\h(U)$ contains no equilibrium
 and intersects both $W_+^\s$ and $W_-^\u$.
Obviously, $x^\h(U)$ is a one-dimensional complex manifold with boundary.
We take $\Gamma=x^\h(U)\cup W_+^\s\cup W_-^\u$
 and the inclusion map as immersion $\iota:\Gamma\to\Cset^2$.
See Fig.~\ref{fig:2a}.
If $x_+=x_-$ and $x^\h(t)$ is a homoclinic orbit,
 then small modifications are needed in the definitions of $\Gamma$ and $\iota$.
Let $0_\pm\in\Gamma$ denote points corresponding to the equilibria $x_\pm$.
Taking three charts, $W_\pm^{\s,\u}$ and $x^\h(U)$, 
 we rewrite the linear system \eqref{eqn:vehl} on $\Gamma$ as follows.

In $x^\h(U)$ we use the complex variable $t\in U$ as the coordinate
 and rewrite \eqref{eqn:vehl} as
\begin{equation}
\frac{\d\xi}{\d t}=J\D^2H(\iota(t))\xi+\hat{g}(\iota(t),\eta_+,\eta_-),\quad
\dot{\eta}_\pm=\pm i\omega\eta_\pm.
\label{eqn:ve1}
\end{equation}
which has no singularity there.
In $W_+^\s$ and $W_-^\u$
 there exist local coordinates $z_+$ and $z_-$, respectively,
 such that $z_\pm(0_\pm)=0$ and $\d/\d t=h_\pm(z_\pm)\d/\d z_\pm$,
 where $h_\pm(z_\pm)=\mp\lambda_\pm z_\pm+O(|z_\pm|^2)$ are holomorphic functions near $z_\pm=0$.
We use the coordinates $z_\pm$ and rewrite \eqref{eqn:vehl} as
\begin{equation}
\frac{\d\xi}{\d z_\pm}=\frac{1}{h_\pm(z_\pm)}J\D^2H(\iota(z_\pm))\xi+\hat{g}_\ell(z_\pm)\eta_\ell,\quad
\frac{\d\eta_\pm}{\d z_\pm}=\frac{i\ell\omega}{h_\pm(z_\pm)}\eta_\ell,
\label{eqn:ve2}
\end{equation}
which have regular singularities at $z_\pm=0$.
Let $M_\pm$ be monodromy matrices of the linear system consisting
 of \eqref{eqn:ve1} and \eqref{eqn:ve2} on $\Gamma$ around $z_\pm=0$.

Let
\[
C_\pm=
\begin{pmatrix}
B_\pm & B_\pm b_\pm\\[0.5ex]
0 & 1
\end{pmatrix},
\]
and let
\begin{equation*}
C_0=
C_+^{-1}C_-=
\begin{pmatrix}
B_+^{-1} & -b_+\\
0 & 1
\end{pmatrix}
\begin{pmatrix}
B_- & B_-b_-\\
0 & 1
\end{pmatrix}
=\begin{pmatrix}
B_0 & B_0b_--b_+\\
0 & 1
\end{pmatrix},
%\label{eqn:c0}
\end{equation*}
where $B_0=B_+^{-1}B_-$.
Recall \eqref{eqn:c}.

\begin{lem}
\label{lem:4a}
The monodromy matrices can be expressed as
\[
M_+=C_0^{-1}
\begin{pmatrix}
\id_2 & (e^{2\pi\ell\omega/\lambda_+}-1)c_+\\
0 & e^{2\pi\ell\omega/\lambda_+}
\end{pmatrix}C_0
\]
and
\begin{equation}
M_-=
\begin{pmatrix}
\id_2 & (e^{-2\pi\ell\omega/\lambda_-}-1)c_-\\
0 & e^{-2\pi\ell\omega/\lambda_-}
\end{pmatrix}.
\label{eqn:M-}
\end{equation}
\end{lem}

\begin{proof}
Let $\tilde{\Psi}(t)=\Psi(t)C_-$.
Then by Lemma~\ref{lem:3d}
 $\tilde{\Psi}(t)$ is a fundamental matrix of \eqref{eqn:veh} such that
\[
\lim_{t\to-\infty}\tilde{\Psi}(t)\Phi_-(-t)=\id_3\quad\mbox{and}\quad
\lim_{t\to+\infty}\tilde{\Psi}(t)C_0\Phi_+(-t)=\id_3.
\]
For the linear system consisting \eqref{eqn:ve1} and \eqref{eqn:ve2} on $\Gamma$,
 we take a fundamental matrix corresponding to $\tilde{\Psi}(t)$.
Since by \eqref{eqn:Phi} its analytic continuation yields the monodromy matrices
\[
\begin{pmatrix}
Q_\pm & 0\\
0 &1
\end{pmatrix}
\begin{pmatrix}
\id_2 & (e^{\pm 2\pi\ell\omega/\lambda_\pm}-1)Q_\pm^{-1}c_\pm\\[1ex]
0 & e^{\pm2\pi\ell\omega/\lambda_\pm}
\end{pmatrix}
\begin{pmatrix}
Q_\pm^{-1} & 0\\
0 &1
\end{pmatrix}
\]
along small loops around $0_\pm$,
 we choose the base point near $0_-$ to obtain the desired result.
\end{proof}

\begin{proof}[Proof of Theorem~$\ref{thm:main}$]
From Lemma~\ref{lem:4a} we have
\begin{equation}
M_+=\begin{pmatrix}
\id_2 & (e^{2\pi\ell\omega/\lambda_+}-1)(B_0^{-1}(b_++c_+)-b_-)\\[0.5ex]
0 & e^{2\pi\ell\omega/\lambda_+}
\end{pmatrix}
\label{eqn:M+}
\end{equation}
since
\[
C_0^{-1}=
\begin{pmatrix}
B_0^{-1} & B_0^{-1}b_+-b_-\\[0.5ex]
0 & 1
\end{pmatrix}.
\]

Suppose that $M_+$ and $M_-$ are commutative and that $\hat{M}_\ell\neq 0$.
From \eqref{eqn:M-} and \eqref{eqn:M+} we have
\[
B_0^{-1}(b_++c_+)-(b_-+c_-)=0,
\]
which yields
\begin{equation}
B_0^{-1}
\begin{pmatrix}
\xi_{+1}\chi_+\\
\xi_{+2}\chi_-+\xi_{+1}^{-1}
\end{pmatrix}m_+
-\begin{pmatrix}
\xi_{-1}\chi_-\\
\xi_{-2}\chi_-+\xi_{-1}^{-1}
\end{pmatrix}m_-
=0
\label{eqn:4a}
\end{equation}
by \eqref{eqn:b}.
Taking the indefinite integral \eqref{eqn:Y} such that $m_-=0$,
 we have $m_+=\hat{M}_\ell$ by \eqref{eqn:m}.
This contradicts \eqref{eqn:4a}.
Hence, if $\hat{M}_\ell\neq 0$,
 then $M_+$ and $M_-$ are not commutative.
We notice that the differential Galois group contains the monodromy group
 and use Theorem~\ref{thm:MR} and Lemmas~\ref{lem:3a} and \ref{lem:3b}
 to complete the proof.
\end{proof}

\begin{rmk}\
\label{rmk:4a}
\begin{enumerate}
\setlength{\leftskip}{-1.8em}
\item[(i)]
Our approach can apply to other time dependency of the perturbations.
For instance, let
\[
g(x,\theta)=\tilde{g}(x)\cn\left(\frac{t}{\sqrt{1-2k^2}}\right)+\tilde{g}_0(x)
\]
where $\cn$ is the Jacobi elliptic function
 with the elliptic modulus $k=\epsilon/\sqrt{2(1+\epsilon^2)}$.
Since $w_2=\epsilon\cn(t/\sqrt{1-2k^2})$ satisfies 
\[
\dot{w}_1=-w_2-w_2^3,\quad
\dot{w}_2=w_1,
\]
where $w_1=\epsilon(\d/\d t)\cn(t/\sqrt{1-2k^2})$, we have
\begin{equation}
\begin{split}
&
\dot{x}=J\D H(x^\h(t))+\epsilon\tilde{g}_0(x)+\tilde{g}_1(x)w_2,\\
&
\dot{\epsilon}=0,\quad
\dot{w}_1=-w_2-w_2^3,\quad
\dot{w}_2=w_1,\end{split}
\label{eqn:rmk4a}
\end{equation}
instead of \eqref{eqn:rsys}.
The system \eqref{eqn:rmk4a} has a solution
 $(x,\epsilon,w_1,w_2)=(x^\h(t),0,0,0)\in\Rset^2\times\Rset\times\Rset\times\Rset$
  and its VE along the solution is given by
\[
\dot{\xi}=J\D^2H(x^\h(t))\xi+\epsilon\tilde{g}_0(x)+\tilde{g}_1(x)\zeta_2,\quad
\dot{\epsilon}=0,\quad
\dot{\zeta}_1=-\zeta_2,\quad
\dot{\zeta}_2=\zeta_1,
\]
which is the same as the VE of \eqref{eqn:asys} along the solution
 $(x,\epsilon,u_1,v_1)=(x^\h(t),0,0,0)$ with $N=1$ and $\omega=1$.
So we can apply the arguments of Sections~$3$ and $4$ to this case.
\item[(ii)]
Suppose that $x^\h(t)$ is such a homoclinic orbit as stated in Remark~$\ref{rmk:1a}$.
Then we can apply the result of Morales-Ruiz {\rm\cite{M02}}
 with the assistance of Ayoul and Zung {\rm\cite{AZ10}}
 to prove the complex-meromorphically nonintegrability of \eqref{eqn:rsys}
 in the sense of Theorem~$\ref{thm:main}$
 if $\chi(t)$ is meromorphic on $\Cset$
 but it is not rational function of $e^{\lambda t}$ and $e^{i\omega t}$,
 and the integral
\[
\int\D H(x^\h(t))\cdot g(x^\h(t),\omega t+\theta)\d t
\]
does not belong to $\Cset(e^{\lambda t},e^{i\omega t},\chi(t))$
 instead of the condition that $M(\theta)$ is constant.
See Assumptions $(2)$ and $(3)$ in {\rm\cite{M02}}.
\end{enumerate}
\end{rmk}

% **********************************************************
% Section 5
% **********************************************************

\section{Examples}

In this section we illustrate our theory
 for the two examples \eqref{eqn:ex1} and \eqref{eqn:ex2}.

\subsection{System~\eqref{eqn:ex1}}

\begin{figure}
\includegraphics[scale=0.25]{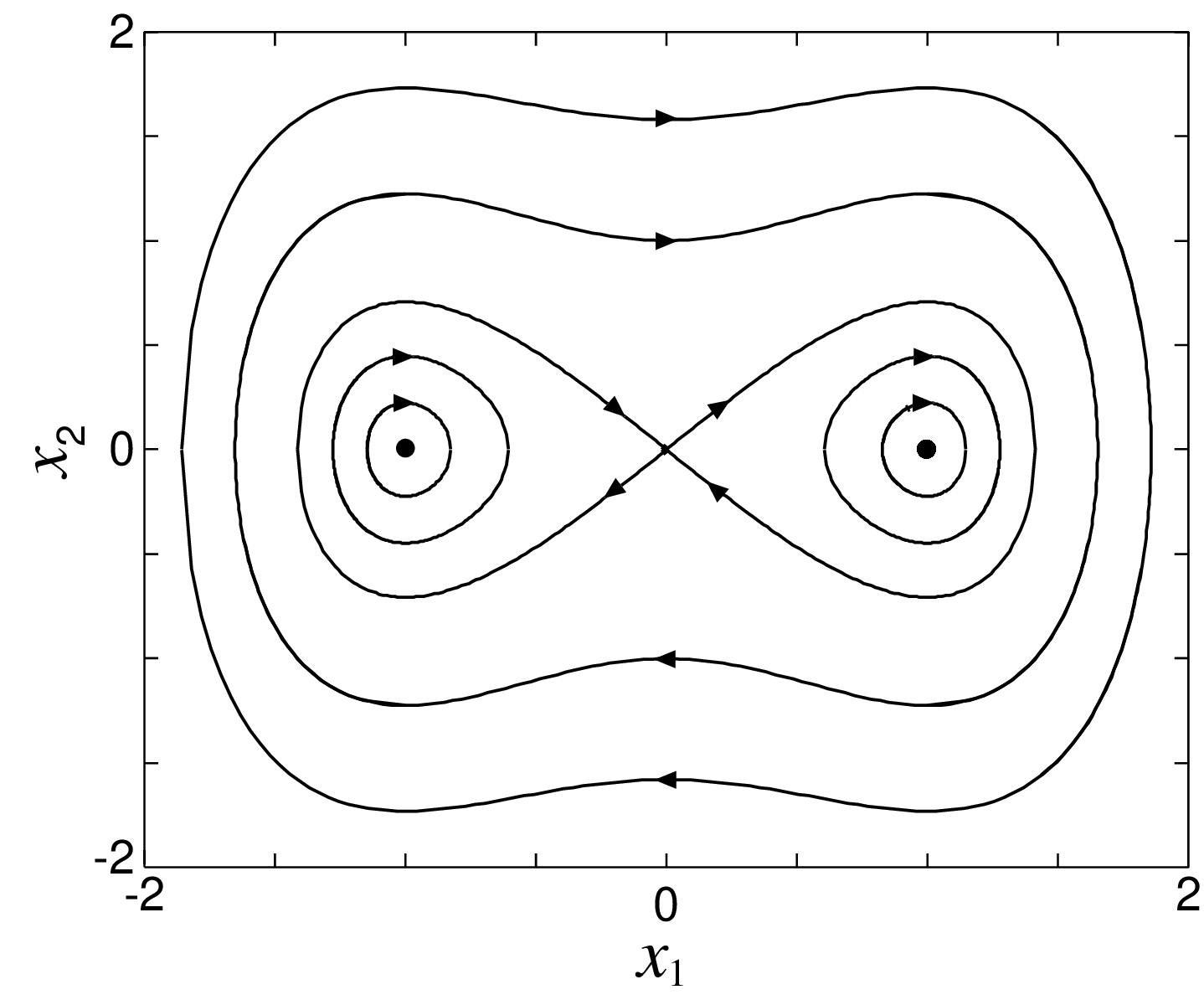}
\caption{Phase portrait of \eqref{eqn:ex1} with $\epsilon=0$.
\label{fig:5a}}
\end{figure}

Consider the system~\eqref{eqn:ex1}.
The unperturbed system is Hamiltonian with the Hamiltonian
\[
H(x)=-\tfrac{1}{2}x_1^2+\tfrac{1}{4}x_1^4+\tfrac{1}{2}x_2^2.
\]
We easily see that assumptions~(A1)-(A3) hold with $x_+=x_-=(0,0)$,
 $\lambda_\pm=1$, $N=1$ and
\begin{equation}
\hat{g}_1(x),\hat{g}_{-1}(x)=(0,\tfrac{1}{2}\beta)^\T,\quad
\hat{g}_0(x)=(0,-\delta x_2)^\T,
\label{eqn:g}
\end{equation}
where the superscript `{\scriptsize$T$}' represents the transpose operator.
See Fig.~\ref{fig:5a} for a phase portrait of \eqref{eqn:ex1} with $\epsilon=0$.
We write \eqref{eqn:rsys} as
\begin{equation}
\begin{split}
&
\dot{x}_1=x_2,\quad
\dot{x}_2=x_1-x_1^3-\epsilon\delta x_2+\beta u_1,\\
&
\dot{\epsilon}=0,\quad
\dot{u}_1=-\omega v_1,\quad
\dot{v}_1=\omega u_1,
\end{split}
\label{eqn:aex1}
\end{equation}
and compute the Melnikov function \eqref{eqn:Mel} for \eqref{eqn:ex1h} as
\begin{align*}
M(\theta)
=&\int_{-\infty}^\infty x_2^\h(t)(\beta\cos(\omega t+\theta)-\delta x_2^\h(t))\d t\\
=&-8\delta\pm 2\pi\beta\sech\biggl(\frac{\pi\omega}{2}\biggr)\cos\theta.
\end{align*}
Applying Theorem~\ref{thm:main}, we obtain the following result.

\begin{prop}
\label{prop:ex1}
The system~\eqref{eqn:aex1} is not real-meromorphically integrable
 near $\hat{\Gamma}=(\{x^\h(t)\mid t\in\Rset\}
 \cup\{0\})\times\{(\epsilon,u_2,v_2)=(0,0,0)\in\Rset\times\Rset\times\Rset\}$.
\end{prop}

\begin{rmk}
If $\beta/\delta>(4/\pi)\cosh(\pi\omega/2)$,
 then $M(\theta)$ has a simple zero,
 so that there exist transverse homoclinic orbits
 and chaotic motions occurs in \eqref{eqn:ex1} and \eqref{eqn:aex1},
 as stated in Section~$1$.
Proposition~$\ref{prop:ex1}$ means that
 the system \eqref{eqn:aex1} is Bogoyavlenskij-nonintegrable
 even when it does not exhibit chaotic dynamics.  
\end{rmk}

\subsection{System~\eqref{eqn:ex2}}

\begin{figure}
\includegraphics[scale=0.45]{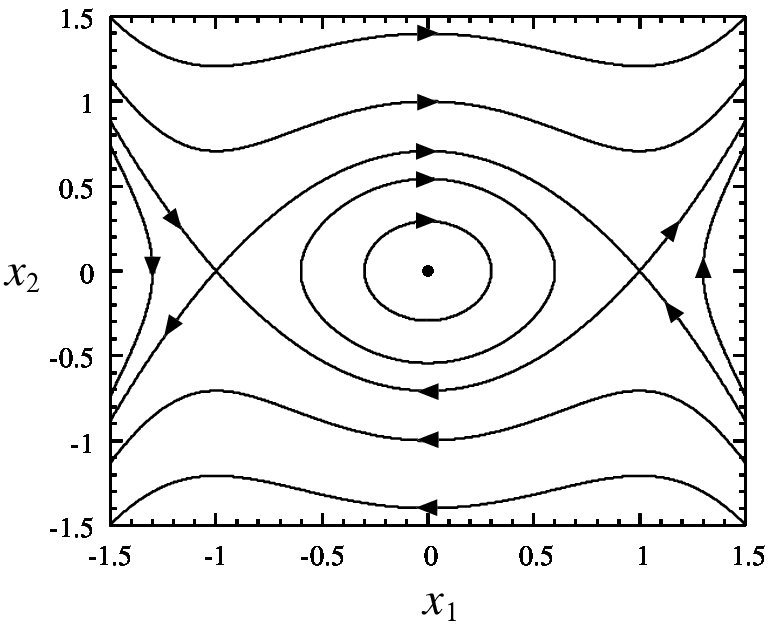}
\caption{Phase portrait of \eqref{eqn:ex2} with $\epsilon=0$.
\label{fig:5b}}
\end{figure}

Consider the system~\eqref{eqn:ex2}.
The unperturbed system is Hamiltonian with the Hamiltonian
\[
H(x)=\tfrac{1}{2}x_1^2-\tfrac{1}{4}x_1^4+\tfrac{1}{2}x_2^2.
\]
We easily see that assumptions~(A1)-(A3) hold with $x_+=(\pm 1,0)$, $x_-=(\mp1,0)$,
 $\lambda_\pm=1$, $N=1$ and \eqref{eqn:g}.
See Fig.~\ref{fig:5b} for a phase portrait of \eqref{eqn:ex1} with $\epsilon=0$.
We write \eqref{eqn:rsys} as
\begin{equation}
\begin{split}
&
\dot{x}_1=x_2,\quad
\dot{x}_2=-x_1+x_1^3-\epsilon\delta x_2+\beta u_1,\\
&
\dot{\epsilon}=0,\quad
\dot{u}_1=-\omega v_1,\quad
\dot{v}_1=\omega u_1,
\end{split}
\label{eqn:aex2}
\end{equation}
and compute the Melnikov function \eqref{eqn:Mel} for \eqref{eqn:ex2h} as
\begin{align*}
M(\theta)
=&\int_{-\infty}^\infty x_2^\h(t)(\beta\cos(\omega t+\theta)-\delta x_2^\h(t))\d t\\
=&-\frac{2\sqrt{2}}{3}\delta
 \pm \sqrt{2}\pi\omega\beta\cosech\biggl(\frac{\pi\omega}{\sqrt{2}}\biggr)\cos\theta.
\end{align*}
Applying Theorem~\ref{thm:main}, we obtain the following result.

\begin{prop}
\label{prop:ex2}
The system~\eqref{eqn:aex2} is not real-meromorphically integrable near $\hat{\Gamma}$.
\end{prop}

\begin{rmk}
\label{rmk:5b}
If $\beta/\delta>(2/3\pi\omega)\sinh(\pi\omega/\sqrt{2})$,
 then $M(\theta)$ has a simple zero,
 so that there exist transverse heteroclinic orbits
 from a periodic orbit near $x=x_-$ to one near $x=x_+$ and vice versa,
 i.e., transverse heteroclinic cycles,
 which indicate chaotic motions in \eqref{eqn:ex2} and \eqref{eqn:aex2} 
 $($see, e.g., Section~$26.1$ of {\rm\cite{W03})},
 as in \eqref{eqn:ex1} and \eqref{eqn:aex1}.
Proposition~$\ref{prop:ex2}$ means that
 the system \eqref{eqn:aex2} is Bogoyavlenskij-nonintegrable
 even when it does not exhibit chaotic dynamics.  
\end{rmk}

\subsection{System~\eqref{eqn:ex3}}

\begin{figure}
\includegraphics[scale=0.45]{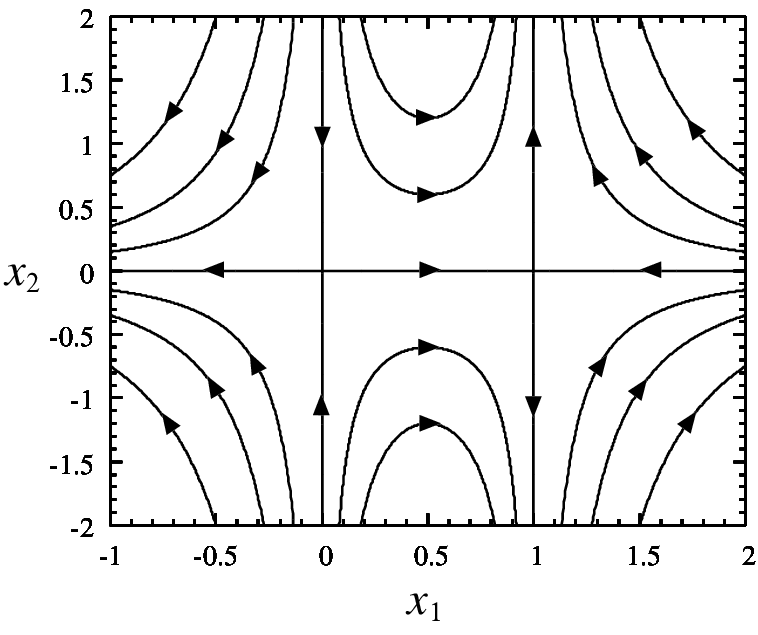}
\caption{Phase portrait of \eqref{eqn:ex2} with $\epsilon=0$.
\label{fig:5c}}
\end{figure}

Consider the system~\eqref{eqn:ex3}.
The unperturbed system is Hamiltonian with the Hamiltonian
\[
H(x)=x_1x_2(1-x_1).
\]
We easily see that assumptions~(A1)-(A3) hold with $x_+=(1,0)$, $x_-=(0,0)$,
 $\lambda_\pm=1$, $N=1$ and \eqref{eqn:g}.
See Fig.~\ref{fig:5c} for a phase portrait of \eqref{eqn:ex1} with $\epsilon=0$.
We write \eqref{eqn:rsys} as
\begin{equation}
\begin{split}
&
\dot{x}_1=x_1-x_1^2,\quad
\dot{x}_2=-(1+\delta)x_2+2x_1x_2+\beta u_1,\\
&
\dot{\epsilon}=0,\quad
\dot{u}_1=-\omega v_1,\quad
\dot{v}_1=\omega u_1,
\end{split}
\label{eqn:aex3}
\end{equation}
and compute the Melnikov function \eqref{eqn:Mel} for \eqref{eqn:ex2h} as
\begin{align*}
M(\theta)
=&\int_{-\infty}^\infty(x_1^\h(t)-x_1^\h(t)^2)(\beta\cos(\omega t+\theta)-\delta x_2^\h(t))\d t\\
=&\pi\omega\cosech\pi\omega\,\cos\theta.
\end{align*}
Applying Theorem~\ref{thm:main}, we obtain the following result.

\begin{prop}
\label{prop:ex3}
The system~\eqref{eqn:aex3} is not real-meromorphically integrable near $\hat{\Gamma}$.
\end{prop}

\begin{rmk}
Since $\beta>0$, $M(\theta)$ has a simple zero
 so that there exist transverse heteroclinic orbits
 from a periodic orbit near $x=x_-$ to one near $x=x_+$,
 as stated in Remark~$\ref{rmk:5b}$,
 but there do not exist transverse heteroclinic cycles.
Thus, the occurrence of chaotic dynamics is not guaranteed
 in \eqref{eqn:ex3} and \eqref{eqn:aex3},
 but Proposition~$\ref{prop:ex3}$ means that
 the system \eqref{eqn:aex3} is Bogoyavlenskij-nonintegrable.
\end{rmk}

\section*{Acknowledgements}
This work was partially supported by the JSPS KAKENHI Grant Number JP22H01138.

%\section*{Data Availability}
%Data sharing not applicable to this article
% as no datasets were generated or analyzed during the current study.
 
% **********************************************************
% Appendices
% **********************************************************

%\appendix
%\renewcommand{\theequation}{\Alph{section}.\arabic{equation}}

% **********************************************************
% Bibliography
% **********************************************************

\end{document}